\documentclass[reqno]{amsart}
\usepackage{graphicx, amsthm, amsmath, amsfonts, amssymb, mathrsfs, color, stmaryrd, bbm} 
\usepackage{hyperref}
\hypersetup{hidelinks}
\usepackage{cleveref}
\usepackage[margin=1in]{geometry}
\usepackage{caption} 
\usepackage{float}
\usepackage{array}
\usepackage{colortbl}
\usepackage{multirow}
\usepackage{enumerate}

\newcommand{\scF}{\mathcal{F}}
\newcommand{\scH}{\mathscr{H}}

\newcommand{\cM}{\mathcal{M}}

\newcommand{\R}{\mathbb{R}}
\newcommand{\bR}{\mathbb{R}}
\newcommand{\bN}{\mathbb{N}}
\newcommand{\bS}{\mathbb{S}}
\newcommand{\bZ}{\mathbb{Z}}
\newcommand{\Mp}{{\mathcal{M}}_\rho}
\newcommand{\Mz}{{\mathcal{M}}_0}
\newcommand{\tMp}{{\widetilde{\mathcal{M}}}_\rho}
\newcommand{\tMz}{{\widetilde{\mathcal{M}}}_{0,\rho}}
\newcommand{\eps}{\epsilon}

\newcommand{\diam}{\mathrm{diam}}

\newcommand{\vol}{\mathrm{vol}}
\newtheorem{theorem}{Theorem}[section]
\newtheorem{mainthm}{Theorem}

\newtheorem{maincor}[mainthm]{Corollary}
\newtheorem{lemma}[theorem]{Lemma}
\newtheorem{example}[theorem]{Example}
\newtheorem{deff}[theorem]{Definition}
\newtheorem{corollary}[theorem]{Corollary}
\newtheorem{proposition}[theorem]{Proposition}
\newtheorem{claim}{Claim}
\theoremstyle{definition}
\newtheorem{remark}[theorem]{Remark}
\newtheorem*{question}{Question}

\title{A Smooth Intrinsic Flat limit with negative curvature}
\author{Jared Krandel}
\address{Jared Krandel, Simons Laufer Mathematical Sciences Institute, 17 Gauss Way, Berkeley CA 94720, USA}
\email{jared.krandel@stonybrook.edu}

\author{Paul Sweeney Jr.}
\address{Paul Sweeney Jr., Universit\`a di Trento, Dipartimento di Matematica, via Sommarive 14, 38123 Povo di Trento, Italy}
\email{paul.sweeney@stonybrook.edu}

\begin{document}
\begin{abstract}
    In 2014, Gromov asked if nonnegative scalar curvature is preserved under intrinsic flat convergence. Here we construct a sequence of closed oriented Riemannian $n$-manifolds, $n\geq 3$, with positive scalar curvature such that their intrinsic flat limit is a Riemannian manifold with negative scalar curvature.
\end{abstract}
\maketitle

\section{Introduction}
In 2014, Gromov \cite{Gro} asked if nonnegative scalar curvature is preserved under intrinsic flat convergence. In particular, one can ask the following question.

\begin{question}
    Let $\left(M_i^n,g_i\right)$, $i\geq 1$, and $\left(M_\infty,g_\infty\right)$ be closed Riemannian manifolds such that $M_i$ converges to $M_\infty$ in some sense. Assume $\left(M_i,g_i\right)$, for all $i$, have nonnegative scalar curvature. Then does $\left(M_\infty,g_\infty\right)$ also have nonnegative scalar curvature?
\end{question}

Gromov \cite{G} and Bamler \cite{B} have given a positive answer to the above question when the convergence is $C^0$-convergence. To show this result, Gromov studied the Dirac operator and a generalized Plateau problem in Riemannian manifolds with corners. Bamler studied the evolution of scalar curvature under Ricci flow to provide a different proof. Lee and Topping showed that the answer to the question is negative when the convergence is uniform convergence of distance functions or when the convergence is Gromov--Hausdorff $\left(\mathrm{GH}\right)$ convergence \cite{LT}. Explicitly, Lee and Topping construct a sequence of Riemannian metrics on $\bS^n$, $n\geq 4$, with positive scalar curvature that converges to a Riemannian metric which is conformal to the unit round metric on $\bS^n$ and has negative scalar curvature.

In this note, we study the question in the setting of Sormani--Wenger intrinsic flat $\left(\scF\right)$ convergence \cite{SW}. We construct a sequence of Riemannian manifolds with positive scalar curvature which converges in the $\scF$-sense to a Riemannian manifold with negative scalar curvature. This provides an answer to Gromov's question from \cite{Gro}. In comparison with Lee and Topping's result \cite{LT}, our limit manifold need not be conformal to the unit round sphere and we produce examples in dimensions $n\geq3.$ Specifically,

\begin{mainthm}\label{t: main}
    Let $n\geq 3$ and $g$ be a Riemannian metric on the standard smooth $n$-sphere with negative Ricci curvature. Then there exists a $\kappa>0$ and a sequence $\left(M^n_j,g_j\right)$ of closed oriented Riemannian $n$-manifolds with scalar curvature $R_{g_j}> \kappa$ such that the corresponding integral current spaces, $\left(M_j, d_{g_j}, \left\llbracket M_j\right\rrbracket\right)$, converge in the $\scF$-sense to $\left(\bS^n, d_{g}, \left\llbracket\bS^n\right\rrbracket\right)$, the integral current space associated to $\left(\bS^n,g\right)$.
\end{mainthm}

 In fact, \Cref{t: main} follows as a corollary of the following theorem.
\begin{mainthm}\label{t: mainbilip} \label{t: bilip}
    Let $\kappa\in\R$, $n\geq 3$, and $\left(M^n,g\right)$ be a closed oriented manifold with scalar curvature $R_g\geq\kappa$ and for $c\in\left(0,1\right)$ let $d:M\times M \to \left[0,\infty\right)$ be a distance function on $M$ such that
    \[
    cd_g\left(x,y\right)\leq d\left(x,y\right)< d_g\left(x,y\right)
    \]
   where $d_g$ is the induced distance function from the Riemannian metric. Then there exists a sequence of Riemannian manifolds $\left(M^n_j,g_j\right)$ with scalar curvature $R_{g_j}\geq \kappa -\frac{1}{j}$  
   such that
     \begin{equation*}
    d_{\scF}\left(\left(M_j,d_{g_j},\left\llbracket M_j\right\rrbracket\right), \left(M,d,\left\llbracket M\right\rrbracket\right)\right)\xrightarrow{j\rightarrow\infty}0.
\end{equation*}
\end{mainthm}
In the statement of \Cref{t: mainbilip}, the inequality hypothesis can be equivalently phrased as the identity map from $\left(M,d_g\right)$ to $\left(M,d\right)$ is distance decreasing and bilipschitz. In \Cref{sec: Bilip}, we prove a slight generalization of \Cref{t: mainbilip} where one can take any distance decreasing biLipschtiz map from $\left(M,d_g\right)$ to a metric space $\left(X,d\right)$ and get the analogous result. 

Another corollary states that for the collection of closed oriented Riemannian $n$-manifolds, $n\geq3$, the subcollection of closed oriented Riemannian $n$-manifolds, $n\geq3$, with scalar curvature $R_{g_j}\geq -1$ is dense in the intrinsic flat topology.
\begin{maincor}\label{c: density}
    Let $n\geq 3$ and $\left(M^n,g\right)$ be a closed oriented Riemannian $n$-manifold. Then there exists a sequence $\left(M^n_j,g_j\right)$ of closed oriented manifolds with scalar curvature $R_{g_j}\geq -1$ that converge in the $\scF$-sense to the integral current space associated to the Riemannian manifold $\left(M^n,g\right)$, i.e.,
    \[
     d_{\scF}\left(\left(M_j,d_{g_j},\left\llbracket M_j\right\rrbracket\right), \left(M,d_{{g}},\left\llbracket M\right\rrbracket\right)\right)\xrightarrow{j\rightarrow\infty}0.
    \]
\end{maincor}

The main technical tool (see \Cref{p :pipe} below) used in the proofs of these results is a generalization of a theorem of Basilio, Kazaras, and Sormani \cite[Theorem 3.1]{BKS}. Our approach allows for any closed oriented Riemannian manifold in contrast to Basilio, Kazaras, and Sormani who only dealt with the standard unit round $n$-sphere. This generalization allows us to prove \Cref{t: mainbilip} from which our other results follow.

The main idea of the proof of \Cref{t: mainbilip} is to add many tunnels to $\left(M^n,g\right)$ which approximates the distances with respect to $d$. This construction is a generalization of the construction in \cite{BKS}, where Basilio, Kazaras, and Sormani constructed a sequence of positive scalar curvature manifolds whose $\scF$-limit is not locally geodesic.  It turns out that their example (\cite[Theorem 1.1]{BKS}) also follows from \Cref{t: mainbilip}. Moreover, to maintain control of the scalar curvature, diameter, and volume of the added tunnels we use a quantitative Gromov--Lawson Schoen--Yau tunnel construction \cite{S}.

\section{Acknowledgements}
The authors would like to thank Christina Sormani for her interest in this result and her many helpful comments and suggestions. The authors thank Raquel Perales for her helpful comments and suggestions. The authors would also like to thank Marcus Khuri and Raanan Schul for their interest in this result. This project has received funding from the European Research Council (ERC) under the European Union’s Horizon 2020 research and innovation programme (grant agreement No. 947923). This paper was partially supported by Simons Foundation International, LTD. This work was supported in part by NSF Grant DMS-2104229 and NSF Grant DMS-2154613. Jared Krandel was partially supported by the National Science Foundation under Grants No. DMS-1763973. This material is also based upon work supported by the National Science Foundation under Grant No. DMS-1928930 while the author was in residence at the Simons Laufer Mathematical Sciences Institute (formerly MSRI) in Berkeley, California during the Fall 2024 Semester.
\section{Background}
In this section, we will review Sormani--Wenger intrinsic flat distance between two integral current spaces. Sormani and Wenger \cite{SW} defined intrinsic flat distance, which generalizes the notion of flat distance for currents in Euclidean space. To do so they used Ambrosio and Kirchheim's generalization of Federer and Fleming's integral currents to metric spaces. We refer the reader to \cite{AK} for further details about currents in arbitrary metric spaces and to \cite{SW} for further details about integral current spaces and intrinsic flat distance.

Let $\left(Z,d^Z\right)$ be a complete metric space. Denote by $\mathrm{Lip}\left(Z\right)$ and $\mathrm{Lip}_b\left(Z\right)$ the set of real-valued Lipschitz functions on $Z$ and the set of bounded real-valued Lipschitz functions on $Z$. 
\begin{deff}[\cite{AK}, Definition 3.1]
We say a multilinear functional 
\[
T:\mathrm{Lip}_b\left(Z\right)\times [\mathrm{Lip}\left(Z\right)]^m\to \bR
\]
on a complete metric space $\left(Z,d^Z\right)$ is an $m$-dimensional current if it satisfies the following properties.
\begin{enumerate}[i]
    \item Locality: $T\left(f,\pi_1,\ldots,\pi_m\right)=0$ if there exists and $i$ such that $\pi_i$ is constant on a neighborhood of $\{f\neq 0\}$.
    \item Continuity: $T$ is continuous with respect to pointwise convergence of $\pi_i$ such that $\mathrm{Lip}\left(\pi_i\right)\leq 1$.
    \item Finite mass: there exists a finite Borel measure $\mu$ on $X$ such that 
    \begin{equation} \label{fm}
        |T\left(f,\pi_1,\ldots,\pi_m\right)|\leq \prod_{i=1}^m \mathrm{Lip}\left(\pi_i\right) \int_Z |f| d\mu
    \end{equation}
    for any $\left(f,\pi_1,\ldots,\pi_m\right)$.
 \end{enumerate}
\end{deff}
We call the minimal measure satisfying (\ref{fm}) the mass measure of $T$ and denote it $||T||$. We can now define many concepts related to a current. $\mathbf{M}\left(T\right)=||T||\left(Z\right)$ is defined to be the mass of $T$ and the canonical set of a $m$-current $T$ on $Z$ is 
\[
\text{set}\left(T\right)=\left\{p\in Z \text{ }\Big|\text{ } \liminf_{r\to 0} \frac{||T||\left(B\left(p,r\right)\right)}{r^m}>0\right\}.
\]
We note that $||T||$ is mutually absolutely continuous with the Hausdorff measure of $\left(\text{set}\left(T\right),d^Z\right)$.
The boundary of a current $T$ is defined as $\partial T:\mathrm{Lip}_b\left(X\right)\times [\mathrm{Lip}\left(X\right)]^{m-1}\to \bR$, where
\[
\partial T\left(f,\pi_1,\ldots,\pi_{m-1}\right)=T\left(1,f,\pi_1,\ldots,\pi_{m-1}\right).
\]
Given a Lipschitz map $\phi:Z\to Z'$, we can pushforward an $m$-current $T$ on $Z$ to an $m$-current $\phi_\# T$ on $Z'$ by defining
\[
\phi_{\#} T\left(f,\pi_1,\ldots,\pi_m\right) = T\left(f\circ \phi,\pi_1\circ \phi ,\ldots,\pi_m\circ \phi\right).
\]
We note that $||\phi_\# T|| \leq \left(\mathrm{Lip}\, \phi\right)^m \phi_\#||T||$ and  $\phi_\#\left(\partial T\right) = \partial \left(\phi_\#T\right)$.  

A standard example of an $m$-current on $Z$ is given by
\[
\phi_\#\left\llbracket\theta\right\rrbracket\left(f,\pi_1,\ldots,\pi_m\right) =\int_A \left(\theta \circ \phi\right) \left(f\circ \phi\right)d\left(\pi_1\circ \phi\right)\wedge \cdots \wedge d\left(\pi_m\circ \phi\right),
\]
where $\phi:\bR^m\to Z$ is bi-Lipschitz and $\theta\in L^1\left(A,\bZ\right)$.
We say that an $m$-current on $Z$ is integer rectifiable if there is a countable collection of bi-Lipschitz maps $\phi_i: A_i\to X$ where $A_i\subseteq \bR^m$ are precompact Borel measurable with pairwise disjoint images and weight functions $\theta_i\in L^1\left(A_i,\bZ\right)$ such that 
\[
T=\sum_{i=1}^\infty \phi_{i\#} \left\llbracket\theta_i\right\rrbracket. 
\]
Moreover, we say an integer rectifiable current whose boundary is also integer rectifiable is an integral current. We denote the space of integral $m$-currents on $Z$ as $\mathbf{I}_m\left(Z\right)$. We say that the triple $\left(X,d,T\right)$ is an $m$-dimensional integral current space if $\left(X,d\right)$ is a metric space, $T\in\mathbf{I}_m\left(\bar{X}\right)$ where $\bar{X}$ is the metric completion of $X$, and $\text{set}\left(T\right)=X$. 
\begin{example}
    Let $\left(M^n,g\right)$ be a closed oriented Riemannian manifold. Then there is a naturally associated $n$-dimensional integral current space $\left(M,d_g,\left\llbracket M\right\rrbracket\right)$, where $d_g$ is the distance function induced by the metric $g$ and $\left\llbracket M \right  \rrbracket: \mathrm{Lip}_b\left(M\right)\times \left[\mathrm{Lip}\left(M\right)\right]^{n}\to \bR$ is given by 
    \[
    \left\llbracket M \right  \rrbracket=\sum_{i,j} \psi_{i\#} \left\llbracket \mathbbm{1}_{A_{ij}} \right  \rrbracket
    \]
    where we have chosen a smooth locally finite atlas $\left\{\left(U_i,\psi_i\right)\right\}_{i\in\bN}$ of $M$ consisting of positively oriented biLipschitz charts, $\psi_i: U_i\subseteq \bR^n\to M$ and $A_{ij}$ are precompact Borel sets such hat $\psi_i\left(A_{ij}\right)$ have disjoint images for all $i,j$ and cover $M$ $\scH^n$-a.e. In this case $\left|\left|\left\llbracket M \right  \rrbracket\right|\right|=\mathrm{dvol}_g$. We note that $\left(M,d,\left\llbracket M \right  \rrbracket\right)$ is also an integral current space when the identity map from $\left(M,d_g\right)$ to $\left(M,d\right)$ is biLipschitz.
\end{example}
The flat distance between two integral currents $T_1$, $T_2\in\mathbf{I}\left(Z\right)$ is
\[
d^Z_F\left(T_1,T_2\right) = \inf\{\mathbf{M}\left(U\right)+\mathbf{M}\left(V\right)\mid U\in\mathbf{I}_m\left(X\right), V\in\mathbf{I}_{m+1}\left(X\right), T_2-T_1=U+\partial V\}.
\]
The intrinsic flat $\left(\scF\right)$ distance between two integral current spaces $\left(X_1,d_1,T_1\right)$ and $\left(X_2,d_2,T_2\right)$ is
\[
d_\scF\left(\left(X_1,d_1,T_1\right),\left(X_2,d_2,T_2\right)\right) = \inf_Z\{d_F^Z\left(\phi_{1\#}T_1,\phi_{2\#}T_2\right)\},
\]
where the infimum is taken over all complete metric spaces $\left(Z,d^Z\right)$ and isometric embeddings $\phi_1:\left(\bar{X}_1,d_1\right)\to \left(Z,d^Z\right)$ and $\phi_2:\left(\bar{X}_2,d_2\right)\to \left(Z,d^Z\right)$. We note that if $\left(X_1,d_1,T_1\right)$ and $\left(X_2,d_2,T_2\right)$ are precompact integral current spaces such that 
\[
d_\scF\left(\left(X_1,d_1,T_1\right),\left(X_2,d_2,T_2\right)\right)=0
\]
then there is a current preserving isometry between $\left(X_1,d_1,T_1\right)$ and $\left(X_2,d_2,T_2\right)$, i.e., there exists an isometry $f:X_1\to X_2$ whose extension $\bar{f}:\bar{X}_1\to \bar{X}_2$ pushes forward the current: $\bar{f}_\#T_1=T_2$. We say a sequence of $\left(X_j,d_j,T_j\right)$ precompact integral current spaces converges to $\left(X_\infty,d_\infty,T_\infty\right)$ in the $\scF$-sense if
\[
d_\scF\left(\left(X_j,d_j,T_j\right),\left(X_\infty,d_\infty,T_\infty\right)\right)\to 0.
\]
If, in addition, $\mathbf{M}\left(T_i\right)\to \mathbf{M}\left(T_\infty\right)$, then we say $\left(X_j,d_j,T_j\right)$ converges to $\left(X_\infty,d_\infty,T_\infty\right)$ in the volume preserving intrinsic flat $\left(\mathcal{VF}\right)$ sense. Lakzian and Sormani in \cite[Theorem 4.6]{LS} were able to estimate the intrinsic flat distance between two manifolds:
\begin{theorem}\label{t: Lakzian}
Suppose $M^n_1=\left(M^n,g_1\right)$ and $M^n_2=\left(M^n,g_2\right)$ are oriented precompact Riemannian manifolds with diffeomorphic subregions $U_j\subset M^n_j$ and diffeomorphisms $\psi_j:U\to U_j$ such that for all $v\in TU$ we have
\[
\frac{1}{\left(1+\eps\right)^2} \psi^*_1g_1\left(v,v\right) < \psi^*_2g_2\left(v,v\right) <\left(1+\eps\right)^2\psi^*_1g_1\left(v,v\right).
\]
We define the following quantities 
\begin{enumerate}
    \item $ D_{U_j} = \sup \{\diam_{M_j}{\left(W\right)}: W \text{ is a component of } U_j\}$.
    \item Let $a\in\bR_+$ be such that $a>\frac{\arccos\left({1+\eps}\right)^{-1}}{\pi} \max\{D_{U_1},D_{U_2}\}$.
    \item $\lambda = \sup_{x,y\in U}|d_{M_1}\left(\psi_1\left(x\right),\psi_1\left(y\right)\right) -d_{M_2}\left(\psi_2\left(x\right),\psi_2\left(y\right)\right)|$.
    \item $h=\sqrt{\lambda\left(\max\{D_{U_1},D_{U_2}\}+\frac{\lambda}{4}\right)}$.
    \item $\bar{h}=\max\left\{h,\sqrt{\eps^2+2\eps}D_{U_1},\sqrt{\eps^2+2\eps}D_{U_2}\right\}.$
\end{enumerate}

Then the intrinsic flat distance between $M^n_1$ and $M^n_2$ is bounded:
\begin{align*}
    d_\scF\left(M_1,M_2\right) &\leq \left(2\bar{h}+a\right)\left(\vol_m\left(U_1\right) + \vol_m\left(U_2\right) + \vol_{m-1}\left(\partial U_1\right)+\vol_{m-1}\left(\partial U_2\right)\right)\\
    &\qquad + \vol_m\left(M_1\setminus U_1\right) +\vol_m\left(M_2\setminus U_2\right).
\end{align*}
\end{theorem}

Lastly, we want to record the following theorem from Huang, Lee, and Sormani \cite[Theorem A.1]{HLS}.
\begin{theorem}\label{t: HLS}
    Fix a precompact $n$-dimensional integral current space $\left(X,d_0,T\right)$ without boundary \\ $\left(\text{e.g. } \partial T=0\right)$ and fix $\alpha>0.$ Suppose that $d_j$ are metrics on $X$ such that 
    \begin{equation}\label{e: HLS}
         \alpha \geq \frac{d_j\left(x,y\right)}{d_0\left(x,y\right)} \geq \frac{1}{\alpha}.
    \end{equation}
    Then there exists a subsequence, also denoted $d_j$, and a metric $d_\infty$ satisfying \eqref{e: HLS} such that $d_j$ converges uniformly to $d_\infty$. That is
    \[
    \eps_j:=\sup\left\{\left|d_j\left(x,y\right)-d_\infty\left(x,y\right)\right|:x,y\in X\right\} \xrightarrow{j\to \infty} 0.
    \]
    Furthermore,
    \begin{align*}
        &\lim_{j\to\infty} d_{\scF}\left(\left(X,d_j,T\right),\left(X,d_\infty,T\right)\right) = 0.
    \end{align*}
    In particular, $\left(X,d_\infty,T\right)$ is an integral current space and $\mathrm{set}\left(T\right)=X$. In fact,
    \[
    d_{\scF}\left(\left(X,d_j,T\right),\left(X,d_\infty,T\right)\right) \leq 2^{\frac{n+3}{2}}\alpha^{n+1}\mathbf{M}_{\left(X,d_0\right)}\left(T\right)\eps_j.
    \]
\end{theorem}

\section{Proof of the Main Results}
The pipe filing technique was first used by Sormani in the appendix of \cite{SW} and was clarified and expanded upon in \cite{BKS}. Here we will use this technique to expand upon the construction in \cite{BKS}. First, we introduce the improved Gromov--Lawson Schoen--Yau tunnels from \cite{S} which will be used to improve the construction in \cite{BKS}.

\begin{proposition}
[Constructing Tunnels]\label{p: propT}
Let $\left(M^n,g\right)$, $n\geq 3$ be a Riemannian manifold with scalar curvature $R_{g}$. Let $\kappa\in \bR$, $\ell\geq 0$, $j\geq 1$, $p_i\in M$, $i=1,2$, and $\delta\in \left(0, \frac{1}{10}\min\{\mathrm{inj_1}, \mathrm{inj_2}\}\right)$, where $\mathrm{inj}_i$ is the injectivity radius of $M$ at $p_i$. Assume $R_{g}\geq\kappa$ on $B_i=B^{M}_{p_i}\left(2\delta\right)\subseteq M$ and $B_1\cap B_2 =\emptyset$. Then there exists a complete Riemannian metric $\bar{g}$ on the smooth manifold $P^n$, which is obtained by removing $B_i$ from $M$ and gluing in a cylindrical region $\left(T_{\delta,\ell},g_{\delta,\ell}\right)$ diffeomorphic to $\bS^{n-1}\times[0,1]$, i.e
\[
P=\left(M\setminus \left(B_1\cup B_2\right)\right) \sqcup T_{\delta,\ell},
\]
such that the following properties are satisfied.
\begin{enumerate}
  
    \item The metrics $g_i$ agree with $\bar{g}$ on $M\setminus B_i$. 
    \item \label{e: con2}There exists constant $A>0$ independent of $\delta$ and $\ell$ such that
\[
\ell<\diam{\left({T_{\delta,\ell}}\right)}<A\delta +\ell \quad \text{and} \quad \vol{\left({T_\delta}\right)}<A\left(\delta^n+\ell\delta^{n-1}\right).
\]
\item $P$ has scalar curvature $R_{\bar{g}}>\kappa -\frac{1}{j}$.
\end{enumerate}
\end{proposition}

\subsection{Definition of $\Mp$, $\tMp$, $\Mz$, and $\tMz$} \label{subsec: definiton of ics} Now we will define four integral current spaces $\Mp$, 
$\tMp$, $\Mz$, and $\tMz$ used in the proof of \Cref{t: 
mainbilip}. Roughly speaking, $\Mp$ is $M$ with two small 
balls removed and a tunnel glued in, $\tMp$ is $M$ with 
two small balls removed, added annular buffer regions in 
which the metric $\widetilde{g}$ transitions to the round 
metric, and a tunnel connecting the round boundaries. 
Lastly, $\Mz$ and $\tMz$ correspond to copies of 
$\left(M,g\right)$ and $\left(M,\widetilde{g}\right)$ 
with strings in place of tunnels.

Let $\left(M^n,g\right)$ be a closed oriented Riemannian manifold with scalar curvature $R_g\geq \kappa$. Let $p,q\in M$ and $\rho\leq \mathrm{inj^M}$. Thus, $B^M_p \left(\rho\right), B^M_q\left(\rho\right)$ are geodesic balls. Now define 
\begin{equation}
    W=M \setminus \left(B^M_p\left(\rho\right)\cup B^M_q\left(\rho\right)\right)
\end{equation}
Now we want to glue in small geodesic balls from the unit round $n$-sphere. Let $g^m_\tau$ be the round metric on the $m$-sphere of radius $\tau$. 
 
Recall that the metric $g$ on $B^M_p\left(\rho\right)\setminus B^M_p\left(\frac{9}{10}\rho\right)$ can be expressed via polar coordinates in the form $dr^2+g_r$ for $\frac{9}{10}\rho\leq r \leq\rho$. Let $\phi:[0,1]\to[0,1]$ be a smooth increasing function such that 
\[
\phi\left(t\right)=\begin{cases}
    0 & t\in\left[0,\frac{1}{4}\right]\\
    1& t\in\left[\frac{3}{4},1\right].
\end{cases}
\]
Define $\psi\left(t\right)=\phi\left(\frac{10}{\rho}\left(t-\frac{9}{10}\rho\right)\right)$ and let $g_{B^M_p\left(\rho\right)}$ be the following smooth Riemannian metric on $B^M_p\left(\rho\right)$
\begin{equation}\label{e: mollified metric p}
g_{B^M_p}=\begin{cases}
   dr^2 + \sin^2\left(r\right)g_1^{n-1} & r\in \left[0,\frac{9}{10}\rho\right]\\
   dr^2 + \psi\left(r\right)\left(g_r- \sin^2\left(r\right)g_1^{n-1}\right) +\sin^2\left(r\right)g_1^{n-1} & r\in \left[\frac{9}{10}\rho,\rho\right],
\end{cases}
\end{equation}
where $dr^2+g_r$ are polar coordinates centered around $p$ in $\left(M,g\right)$. Similarly, we can do the same procedure for $q$. Let $g_{B^M_q\left(\rho\right)}$ be the following smooth Riemannian metric on $B^M_q\left(\rho\right)$.
\begin{equation}\label{e: mollified metric q}
g_{B^M_q}=\begin{cases}
   dr^2 + \sin^2\left(r\right)g_1^{n-1} & r\in \left[0,\frac{9}{10}\rho\right]\\
   dr^2 + \psi\left(r\right)\left(g_r- \sin^2\left(r\right)g_1^{n-1}\right) +\sin^2\left(r\right)g_1^{n-1} & r\in \left[\frac{9}{10}\rho,\rho\right],
\end{cases}
\end{equation}
where $dr^2+g_r$ are polar coordinates centered around $q$ in $\left(M,\widetilde{g}\right)$.
Now we can glue $\left({B^M_p},g_{B^M_p}\right)$  and $\left({B^M_q},g_{B^M_q}\right)$ into $W$ resulting in the smooth manifold
\begin{equation}
    \left({M} = W \sqcup B^M_p\left(\rho\right) \sqcup B^M_q\left(\rho\right),\widetilde{g}\right).
\end{equation}
\begin{remark}
    The metric $\widetilde{g}$ on ${M}$ implicitly depends on $\rho$. 
\end{remark}

Fix $0<\ell<d_g\left(p,q\right)$ and $\delta<\min\left\{\frac{1}{2}\rho,\rho_0\right\}$ for some $\rho_0$ we will choose in \Cref{l: useful}. We do this so that $\ell <  d_{\widetilde{g}}\left({p},{q}\right)$ automatically.  Before we show \Cref{l: useful}, we first recall from \cite[Lemma 2.1]{D} (cf. \cite[Lemma 1]{GL}), the following lemma:
\begin{lemma} \label{l: princurv}
Let $\left(M^n,g\right)$ be a Riemannian manifold. The principal curvatures of the hypersurface $\bS^{n-1}\left(\eps\right)$ in  a geodesic ball $B\subset M$ are each of the form $\frac{1}{-\eps} + O\left(\eps\right)$ for small $\eps$. Furthermore, let $g_\eps$ be the induced metric on $\bS^{n-1}\left(\eps\right)$ and recall $g^n_{\eps}$ is the round metric of radius $\eps$. Then, as $\eps \to 0$, $\frac{1}{\eps^2}g_\eps \to \frac{1}{\eps^2}g_{rd,\eps} =g_{rd}$ in the $C^2$ topology, moreover, one can fix a $C^2$-norm $|\cdot|_2$  where $|g_{rd}-\frac{1}{\eps^2}g_\eps|_2\leq\eps^2$.
\end{lemma}
\begin{lemma} \label{l: useful}
    Consider $\left(M,g\right)$ and $\left(M,\widetilde{g}\right)$ defined above and assume $0<\ell<d_g\left(p,q\right)$. Then there exists $\rho_0>0$ such that for any $0<\rho<\rho_0$ we have $\ell<d_{\widetilde{g}}\left({p},{q}\right)$.
\end{lemma}
\begin{proof}
    Recall that $\widetilde{g}$ implicitly depends on $\rho$. Only in this proof will we highlight this and denote $\widetilde{g}$ as $\widetilde{g}^\rho$. We will show that $\widetilde{g}^\rho$ converges to $g$ in the $C^2$-topology. Specifically, we note on $W$ we have $g=\widetilde{g}^\rho$; therefore, we consider $g$ and $\widetilde{g}^\rho$ on $B^M_p\left(\rho\right)$ and $B^M_q\left(\rho\right)$. On $B^M_p\left(\rho\right)$ we have 
    \begin{align}
        \left|\widetilde{g}^\rho - g\right| = \begin{cases}
            0, &  r\in \left[0,\frac{9}{10}\rho\right]\\
            \left|\psi\left(r\right)\left(g_r- \sin^2\left(r\right)g_1^{n-1}\right) +\sin^2\left(r\right)g_1^{n-1} -g_r\right|,&  r\in \left[\frac{9}{10}\rho,\rho\right].
        \end{cases}
    \end{align}
Therefore, we see that 
\begin{equation}\label{e: g'sclose}
\begin{split}
     \left|\widetilde{g}^\rho - g\right| &\leq
     \left|{\psi\left(r\right)\left(g_r- \sin^2\left(r\right)g_1^{n-1}\right) +\sin^2\left(r\right)g_1^{n-1}} - g_r\right|_2 \\
    &\leq 2\left|\sin^2\left(r\right)g_1^{n-1} - g_r\right|_2\\
   & \leq 2\left(\left|\sin^2\left(r\right)g_1^{n-1} - r^2g_1^{n-1}\right|_2 +\left|r^2g_1^{n-1}-g_r\right|_2\right)\\
   &\leq 2r^2 \left(\left|\frac{\sin^2\left(r\right)}{r^2}g_1^{n-1} - g_1^{n-1}\right|_2 +\left|g_1^{n-1}-\frac{1}{r^2}g_r\right|_2\right)\\
   &\leq A\rho^4,
\end{split}
\end{equation}
where the first inequality follows since $\left|\psi\left(r\right)\right|\leq 1$ and the last inequality follows by \Cref{l: princurv} and the fact $g_r$ are metrics on geodesics spheres of radius less than $\rho$. Likewise, one can show the same inequality on $B^M_q\left(\rho\right)$. Therefore, we have that $d_{\widetilde{g}^\rho}$ converges uniformly to $d_g$. Thus, concluding the proof.
\end{proof}

By \Cref{p: propT}, there is a Riemannian manifold 
\begin{equation}\label{e: 1tunnel}
\left(M_\rho=M \setminus \left(B^M_p\left(\delta\right)\cup B^M_q\left(\delta\right)\right) \cup T_{\delta,\ell}, g_\rho\right).
\end{equation}
Moreover, we choose $j$ in \Cref{p: propT} such that $j:=j\left(\rho\right)\xrightarrow{\rho\to0}\infty$. Thus, the scalar curvature is larger than $\kappa-\frac{1}{j\left(\rho\right)}$. And define the integral current space
\[
\Mp = \left({M}_\rho,d_{{{g}}_\rho},\left\llbracket{{M}_\rho} \right\rrbracket\right).
\]
Again by \Cref{p: propT}, there is a Riemannian metric on
\begin{equation}\label{e: 2tunnel}
\left(\widetilde{{M}}_\rho={M} \setminus \left(B^{{M}}_{{p}}\left(\delta\right)\cup B^{{M}}_{{q}}\left(\delta\right)\right) \cup \widetilde{T}_{\delta,\ell}, \widetilde{g}_\rho\right).
\end{equation}
And define the integral current space
\[
\tMp = \left(\widetilde{{M}}_\rho,d_{{\widetilde{g}}_\rho},\left\llbracket{{\widetilde{M}}_\rho} \right\rrbracket\right).
\]
Now let $\left(M_0,d_0\right)$ denote the metric space obtained by joining $p$ and $q$ in $\left(M,g\right)$ by a line segment $I_\ell=[0,\ell]$ of length $\ell$:
\begin{equation*}\label{e: 1string}
   M_0= M\sqcup_{p\sim 0,q\sim d} I_\ell. 
\end{equation*}
We note that there is natural embedding $\dot{\iota}:M\to M_0$. Define the integral current space
\[
\Mz = \left({M},d_{0},\left\llbracket{{M}} \right\rrbracket\right).
\]
Now let $\left(\widetilde{{M}}_{0,\rho},\widetilde{d}_{0,\rho}\right)$ denote the metric space obtained by joining ${p}$ and ${q}$ in $\left(M,\widetilde{g}\right)$ by a line segment  $I_\ell=[0,\ell]$ of length $\ell$:
\begin{equation*}\label{e: 2string}
   \widetilde{{M}}_{0,\rho}= {M}\sqcup_{{p}\sim 0,{q}\sim d} I_\ell. 
\end{equation*}
We note that as sets $ \widetilde{{M}}_{0,\rho}$ and $M_0$ are the same. Also, there is natural embedding $\dot{\iota}:M\to \widetilde{{M}}_0$. Define the integral current space
\[
\tMz = \left({M},\widetilde{d}_{0,\rho},\left\llbracket{{M}} \right\rrbracket\right).
\]

\subsection{$\Mz$ and $\Mp$ are close.}
The following is the generalization of \cite[Theorem 3.1]{BKS}.
\begin{proposition} \label{p :pipe}
     Let $\Mz$ and $\Mp$ be defined as above.
    Then there exists a constant $A>0$, depending only on $\ell$ and $g$ such that
    \begin{equation}\label{e :pipe}
        d_{\scF}\left(\Mz,\Mp\right) \leq A{\sqrt{\rho}}.
    \end{equation}
\end{proposition}
\begin{proof}
    By the triangle inequality, we have that
    \[
    d_{\scF}\left(\Mp, \Mz\right) \leq d_{\scF}\left(\tMp, \Mp\right) + d_{\scF}\left(\tMp,\tMz\right)+d_{\scF}\left(\tMz, \Mz\right)
    \]
  By our construction, 
  we can apply \cite[Theorem 3.1]{BKS}, to see there is a a constant $A>0$ depending only on $\ell$, $\vol_g\left(M\right)$, and $\diam_g\left(M\right)$ such that
    \begin{equation}\label{e:T1}
        d_{\scF}\left(\tMp, \tMz\right) \leq A\rho.
    \end{equation}
    \begin{claim}\label{c: T1} There is a a constant $A>0$ depending only on $\ell$ and $g$ such that
        \[
        d_{\scF}\left(\tMp, \Mp\right)\leq A\sqrt{\rho}
        \]
    \end{claim}
    \begin{proof}[Proof of \Cref{c: T1}]
    Recall that $W$ is a subset of $\widetilde{{M}}_\rho$, $M_\rho$, and $M$. Moreover, $W$ has the same induced metric in each. Define:
    \begin{align}
    &\lambda = \sup_{x,y\in W}|d_{M_\rho}\left(x,y\right) -d_{\widetilde{{M}}_\rho}\left(x,y\right)|\label{e: lambda}\\
         &h=\sqrt{\lambda\left(\diam\left(W\right)+\frac{\lambda}{4}\right)}\nonumber
    \end{align}
By \cite[Theorem 4.6]{LS} (see \Cref{t: Lakzian} above) and \Cref{p: propT},
    \begin{align*}
    d_{\scF}\left(\tMp, \Mp\right) &\leq \left(2{h}+\rho\right)\left(\vol_g\left(W\right) + \vol_{\widetilde{g}}\left(W\right) + \vol_{g}\left(\partial W\right)+\vol_{\widetilde{g}}\left(\partial \widetilde{W}\right)\right)\\
    &\qquad + \vol_g\left({M}\setminus W\right) +\vol_{\widetilde{g}}\left({M}\setminus \widetilde{W}\right).\\
    &\leq \left(2h+\rho\right) \left(2\vol_g\left(M\right) + A\rho^{n-1}\right) + A\rho^{n-1}.
\end{align*}
  All that is left is to estimate $h$ which we will do by estimating $\lambda$. Let $x,y\in W$ and let $\gamma$ and $\widetilde{\gamma}$ be geodesics between $x,y$ in $M_\rho$ and $\widetilde{{M}}_\rho$ respectively. Let $C=M_\rho \setminus W$ and $\widetilde{C}=\widetilde{{M}}_\rho \setminus W$. 
 Thus,
 \begin{align}\label{e: 1lambda}
     \begin{split}
     d_{M_\rho}\left(x,y\right) -d_{\widetilde{{M}}_\rho}\left(x,y\right) &= \mathrm{length}_{M_\rho}\left(\gamma\right) -\mathrm{length}_{\widetilde{{M}}_\rho}\left(\widetilde{\gamma}\right)\\
     &\leq \mathrm{length}_{M_\rho}\left(\widetilde{\gamma}\right) -\mathrm{length}_{\widetilde{{M}}_\rho}\left(\widetilde{\gamma}\right)\\
     &=\mathrm{length}_{M_\rho}\left(\widetilde{\gamma}\cap C\right) -\mathrm{length}_{\widetilde{{M}}_\rho}\left(\widetilde{\gamma}\cap \widetilde{C}\right),
 \end{split}
 \end{align}
 where the last line follows because $M_\rho$ and $\widetilde{{M}}_\rho$ are isometric on $W$.
 
 Interchanging the roles of $M_\rho$ and $\widetilde{{M}}_\rho$ allows one by an analogous argument to show,
 \begin{equation}\label{e: 2lambda}
   d_{\widetilde{{M}}_\rho}\left(x,y\right) - d_{M_\rho}\left(x,y\right)\leq \mathrm{length}_{\widetilde{{M}}_\rho}\left({\gamma}\cap \widetilde{C}\right) - \mathrm{length}_{M_\rho}\left({\gamma}\cap C\right).
 \end{equation}
Note that in $C$ and $\widetilde{C}$ the metrics will be close to each other in the $C^2$-norm; therefore, \eqref{e: lambda} will be bounded by a function of $\rho$. To show this let's write down the metrics $g_{\widetilde{C}}$ and $g_C$ for $\widetilde{C}$ and $C$, respectively. Using the definitions of $M_\rho$ and $\widetilde{{M}}_\rho$ from Subsection \ref{subsec: definiton of ics} and equations \eqref{e: mollified metric p} and \eqref{e: mollified metric q} we see
\[\label{e: C metric}
    g_C =\left\{\begin{array}{lll}
         dr^2+g_r & r\in [0,\rho-\delta] & \text{corresponds to } B^M_p\left(\rho\right) \setminus B^M_p\left(\delta\right)\\
         dr^2 +g_r & r\in [\rho-\delta, L]  & \text{corresponds to } T_{\delta,\ell}\\
         dr^2+g_r & r\in [L, L+\rho-\delta]& \text{corresponds to } B^M_q\left(\rho\right) \setminus B^M_q\left(\delta\right).
     \end{array}\right.
 \]
 \[
  g_{\widetilde{C}}=\left\{\begin{array}{lll}\label{e: tilde{C} metric}
      
         ds^2+g_s & s\in \left[0,\frac{1}{40}\rho\right]   &\text{corresponds to } B^{\widetilde{M}}_{{p}}\left(\rho\right) \setminus B^{\widetilde{M}}_{{p}}\left(\frac{39}{40}\rho\right)
         \\
           g_{\mathrm{mol}} \left(\bar{s}\right) & s\in \left[\frac{1}{40}\rho, \frac{5}{40}\rho\right]  &  \text{corresponds to } B^{\widetilde{M}}_{{p}}\left(\frac{39}{40}\rho\right) \setminus B^{\widetilde{M}}_{{p}}\left(\frac{35}{40}\rho\right)
           \\
       ds^2 + \sin^2\left(\bar{s}\right)g_1^{n-1} & s\in \left[\frac{5}{40}\rho,\rho -\delta\right]   &\text{corresponds to } B^{\widetilde{M}}_{{p}}\left(\frac{35}{40}\rho\right) \setminus B^{\widetilde{M}}_{{p}}\left(\delta\right)
        \\
         ds^2 +g_s & s\in [\rho-\delta, \widetilde{L}]    &\text{corresponds to } \widetilde{T}_{\delta,\ell}
         \\
          ds^2+\sin^2\left(\hat{s}\right)g_1^{n-1} & s\in \left[\widetilde{L},\widetilde{L}+ \frac{35}{40}\rho - \delta \right] &\text{corresponds to } B^{\widetilde{M}}_{{q}}\left(\frac{35}{40}\rho\right) \setminus B^{\widetilde{M}}_{{q}}\left(\delta\right)
          \\
           g_{\mathrm{mol}}\left(\hat{s}\right) & s\in \left[\widetilde{L}+ \frac{35}{40}\rho -\delta, \widetilde{L}+ \frac{39}{40}\rho -\delta \right]   &\text{corresponds to } B^{\widetilde{M}}_{{q}}\left(\frac{39}{40}\rho\right) \setminus B^{\widetilde{M}}_{{q}}\left(\frac{35}{40}\rho\right)
           \\
            ds^2+g_s & s\in \left[\widetilde{L}+ \frac{39}{40}\rho -\delta,\widetilde{L}+\rho -\delta\right]  &\text{corresponds to } B^{\widetilde{M}}_{{q}}\left(\rho\right) \setminus B^{\widetilde{M}}_{{q}}\left(\frac{39}{40}\rho\right),
     
 \end{array}\right.
 \]
 where $g_{\mathrm{mol}}\left(s\right)={ds^2 + \psi\left(s\right)\left(g_s- \sin^2\left(s\right)g_1^{n-1}\right) +\sin^2\left(s\right)g_1^{n-1}}$, $\bar{s}=\rho-s$, $\hat{s}=s-\widetilde{L}+\delta$. Also, the notation ``$B^{{M}}$" corresponds to balls in $\left(M,{g}\right)$ and ``$B^{\widetilde{M}}$" corresponds to balls in $\left(M,\widetilde{g}\right)$. Now we wish to estimate the $C^2$-norm between $g_C$ and $g_{\widetilde{C}}$.

By construction, on $\left[0,\frac{1}{40}\rho\right]$ $g_C$ and $g_{\widetilde{C}}$ are identical and so are isometric.

On $\left[\frac{1}{40}\rho,\frac{5}{40}\rho\right]$ by using the same computation as \eqref{e: g'sclose} we have
\[
\left|g_C-g_{\widetilde{C}}\right|_2 = \left|{\psi\left(r\right)\left(g_r- \sin^2\left(r\right)g_1^{n-1}\right) +\sin^2\left(r\right)g_1^{n-1}} - g_r\right|_2 \leq A\rho^4.
\]

We note this $A$ is some constant depending on $g$.
On $\left[\frac{5}{40}\rho,\rho-\delta\right]$, again we have by a similar argument and by \Cref{l: princurv}.
\[
 \left|g_C-g_{\widetilde{C}}\right|_2 = \left|g_r- \sin^2\left(r\right)g_1^{n-1} \right|_2 \leq A\rho^4,
\]
for a constant $A$ depending on $g$. On $\widetilde{T}_{\delta,\ell}$ we reparameterize the interval so that it has the same length as $\left[\rho-\delta, L\right]$. Let this reparametrization be 
\[
s\left(r\right)=\frac{\left({L}-\rho+\delta\right)\left(r-\rho+\delta\right)}{\widetilde{L}-\rho+\delta} +\rho-\delta 
\]
and so the metric looks like
\[
\left(\frac{{L}-\rho+\delta}{\widetilde{L}-\rho+\delta}\right)^2dr^2 + g_{s\left(r\right)}.
\]
Thus, using this new parametrization we have on $[\rho-\delta, L]$:
\begin{align*}
     \left|g_C-g_{\widetilde{C}}\right|_2 &\leq  \left|1-\left(\frac{{L}-\rho+\delta}{\widetilde{L}-\rho+\delta}\right)^2\right|\left|dr^2\right|_2 +  \left|g_r-g_{s\left(r\right)}\right|_2.
\end{align*}
Now by \Cref{p: propT} there is a constant $A$ such that $\ell\leq L,\widetilde{L}\leq A\delta +\ell$. Therefore, for some $A$ depending only on $g$ and $\ell$,
\[
\left|1 - \left(\frac{{L}-\rho+\delta}{\widetilde{L}-\rho+\delta}\right)^2\right| = \left| \frac{{\widetilde{L}-L}}{\widetilde{L}-\rho+\delta}\right|\left|1 + \frac{{L}-\rho+\delta}{\widetilde{L}-\rho+\delta}\right|\leq A\rho.
\]
Now by \Cref{l: princurv},
\begin{align*}
    \left|g_r-g_{s\left(r\right)}\right|_2 &\leq \left|g_r-r^2g_1^{n-1}\right|_2 + \left|r^2g_1^{n-1}-s\left(r\right)^2g_1^{n-1}\right|_2 +\left|s\left(r\right)^2g_1^{n-1}-g_{s\left(r\right)}\right|_2\\
    &\leq c\rho^4 + \left|r^2g_1^{n-1}-s\left(r\right)^2g_1^{n-1}\right|_2\\
    &\leq A\rho^2,
\end{align*}
where the last inequality follows because $|s\left(r\right)- r|\leq A\rho^2$ for some constant $A$ depending on $\ell$  and $g$.

By symmetry, similar estimates follow for the remaining three intervals in the piecewise definition of $g_{\widetilde{C}}$ where $A$ depends on $\ell$ and $g$. Therefore, we conclude that for some constant $A$ depending on $\ell$ and $g$
\begin{equation}\label{e: tensor}
    \left|g_C-g_{\widetilde{C}}\right|_2\leq A\rho
\end{equation}
which implies by \eqref{e: 1lambda}, \eqref{e: 2lambda}, and \eqref{e: tensor} that $\lambda \leq A{\rho}$. Therefore,

\begin{equation}\label{e:T2}
    d_{\scF}\left(\tMp, \Mp\right) \leq A\sqrt{\rho}.
\end{equation}
\end{proof}
\begin{claim}\label{c: T2}
    There is a constant $A>0$, depending only on $\ell$ and $g$ so that
    \[
     d_{\scF}\left(\tMz,\Mz\right) \leq A\rho^2.
    \]
\end{claim}
\begin{proof}[Proof of \Cref{c: T2}]
By \cite[Appendix A]{HLS} (see \Cref{t: HLS} above), we have that if there is an $\alpha\geq 1$ independent of $\rho$ such that for all $x,y\in M$
\begin{equation} \label{e: alpha}
    \frac{1}{\alpha}\leq \frac{\widetilde{d}_{0,\rho}\left(x,y\right)}{d_0\left(x,y\right)} \leq \alpha
\end{equation}
and
\begin{equation}\label{e: eps}
    \eps_\rho:= \sup_{\left(x,y\right)\in M\times M}\left\{\left|d_{0,\rho}\left(x,y\right) -d_0\left(x,y\right)\right|\right\} \xrightarrow{\rho\to 0}  0
\end{equation}
then there exists a subsequence which we will abuse notation and call $\tMz$ such that
\[
d_{\scF}\left(\tMz, \Mz\right) \leq A\eps_\rho,
\]
where $A$ depends on $\alpha$ and the mass of $\left(M_0,d_0,\left\llbracket M\right\rrbracket\right)$. 

Recall $\left(M_0,d_0\right)$ and $\left(\widetilde{{M}}_{0,\rho},\widetilde{d}_{0,\rho}\right)$ are geodesic spaces and let $x,y\in M$. We note that $M$ is naturally a subset of $M_0$ and $\widetilde{{M}}_{0,\rho}$. Let $\gamma$ be a geodesic between $x,y$ in $\left(M_0,d_0\right)$ and $\widetilde{\gamma}$ be a geodesic between $x,y$ in $\left(\widetilde{{M}}_{0,\rho},\widetilde{d}_{0,\rho}\right)$. Finally, define $E=M\setminus W$. Now we estimate:
\begin{align*}
    \widetilde{d}_{0,\rho}\left(x,y\right) -d_0\left(x,y\right) &= \mathrm{length}_{\widetilde{{M}}_{0,\rho}}\left(\widetilde{\gamma}\right) - \mathrm{length}_{{M}_{0}}\left(\gamma\right)\\
    &\leq \mathrm{length}_{\widetilde{{M}}_{0,\rho}}\left({\gamma}\right) - \mathrm{length}_{{M}_{0}}\left(\gamma\right)\\
    & = \mathrm{length}_{\widetilde{{M}}_{0,\rho}}\left({\gamma}\cap W\right) + \mathrm{length}_{\widetilde{{M}}_{0,\rho}}\left({\gamma}\cap I_\ell\right) + \mathrm{length}_{\widetilde{{M}}_{0,\rho}}\left({\gamma}\cap E\right)\\
    &\qquad- \mathrm{length}_{{M}_{0}}\left(\gamma\cap W\right) - \mathrm{length}_{{M}_{0}}\left(\gamma\cap I_\ell\right) -\mathrm{length}_{{M}_{0}}\left(\gamma\cap E\right)\\
    &= \mathrm{length}_{\widetilde{{M}}_{0,\rho}}\left({\gamma}\cap E\right)-\mathrm{length}_{{M}_{0}}\left(\gamma\cap E\right).
\end{align*}
By similar arguments to the proof of \Cref{c: T1} we will show that the final line above is less than $A\rho$, where $A$ depends on $g$ and $\ell$.
Specifically, we note that $E$ is a subset $M$. Thus, $\left(E,g|_E\right)$ and $\left(E,\widetilde{g}|_E\right)$ are smooth manifolds with their induced metrics; moreover, the naturally associated metric spaces are isometric to the subsets of $\left(\widetilde{{M}}_{0,\rho},\widetilde{d}_{0,\rho}\right)$ and $\left(M_0,d_0\right)$. Moreover, note that $E=E_1\sqcup E_2$ where $E_1$ and $E_2$ are disjoint balls. Thus, it suffices to study what happens on $E_i$. 
\begin{align}
   \left|g|_{E_i} - g_1^n\right|_2 =
       \left|g_r-\sin^2\left(r\right)g_1^{n-1}\right|_2 \qquad r\in \left[0,\rho\right]
\end{align}
and
\begin{align}
   \left|\widetilde{g}|_{E_i} - g_1^n\right|_2 =
   \left\{\begin{array}{ll}
       0 & r\in \left[0,\frac{9}{10}\rho\right]\\
       \left|\psi\left(r\right)\left(g_r- \sin^2\left(r\right)g_1^{n-1}\right)\right|_2& r\in \left[\frac{9}{10}\rho,\rho\right].
   \end{array}\right.
\end{align}
Now by \Cref{l: princurv} we can conclude that $ \left|g|_{E_i} - g_1^n\right|_2\leq A\rho^4$ and $ \left|\widetilde{g}|_{E_i} - g_1^n\right|_2\leq A\rho^4$ (where $A$ depends on $g$) and so by the triangle inequality
\[
 \left|g|_{E} -\widetilde{g}|_{E}\right|_2 \leq A\rho^4.
\]
Therefore, we may conclude that $\mathrm{length}_{\widetilde{{M}}_{0,\rho}}\left({\gamma}\cap E\right)-\mathrm{length}_{{M}_{0}}\left(\gamma\cap E\right)\leq A\rho^2$ and so
\[
 d_{0,\rho}\left(x,y\right) -d_0\left(x,y\right)\leq A\rho^2,
\]
where $A$ depends on $g$. Likewise, by an analogous argument, we conclude $ d_0\left(x,y\right) -d_{0,\rho}\left(x,y\right) \leq A\rho^2;$ therefore, $\eps_\rho \leq A\rho^2$, where $A$ depends on $g$

Now consider,
\begin{align*}
    \frac{\widetilde{d}_{0,\rho}\left(x,y\right)}{d_0\left(x,y\right)} &= \frac{\mathrm{length}_{\widetilde{{M}}_{0,\rho}}\left(\widetilde{\gamma}\right)}{\mathrm{length}_{{M}_{0}}\left(\gamma\right)}\\
    &\leq \frac{\mathrm{length}_{\widetilde{{M}}_{0,\rho}}\left({\gamma}\right)}{\mathrm{length}_{{M}_{0}}\left(\gamma\right)} \\
&=\frac{\mathrm{length}_{\widetilde{{M}}_{0,\rho}}\left({\gamma}\cap \left(W\cup I_\ell\right)\right) +  \mathrm{length}_{\widetilde{{M}}_{0,\rho}}\left({\gamma}\cap E\right)}{\mathrm{length}_{{M}_{0}}\left(\gamma\cap \left(W\cup I_\ell\right)\right)  +\mathrm{length}_{{M}_{0}}\left(\gamma\cap E\right)}\\
    &\leq  1 + \frac{\mathrm{length}_{\widetilde{{M}}_{0,\rho}}\left({\gamma}\cap E \right)}{\mathrm{length}_{{M}_{0}}\left(\gamma\cap E \right)}, \\
\end{align*}
where in the last inequality we use the fact that $g$ and $\widetilde{g}$ coincide outside $E$ and $\frac{a+b}{a+c} \leq 1 + \frac{b}{c}$ for any $a,b,c \geq 0.$ One can prove this inequality by observing that it suffices to consider the case $a=1$ and rearranging terms to show that it is equivalent to $-c^2 \leq b$.

Now, we note that by compactness of $M$ we have that for all $v\in TM$, there is a constant $A$ depending on $g$ such that $\frac{1}{A} ||v||^2\leq g\left(v,v\right)\leq A||v||^2$ where $||\cdot||$ is the standard inner product on $\bR^n$. Similarly, for all $\rho\in\left(0,1\right)$, we have by the definition of $\widetilde{g}$ that there exists a constant $A_1$ depending on $g$ $\left(\text{and } \sin^2\left(r\right)g_1^{n-1}\right)$ such that $\frac{1}{A_1}||v||^2\leq \widetilde{g}\left(v,v\right)\leq A_1||v||^2$ for all $v\in TM$. Therefore, we have that on $E$
\[
\mathrm{length}_{\widetilde{{M}}_{0,\rho}} \left(\gamma\cap E\right) \leq A_1\int_0^1 ||\dot{\gamma}||dt
\]
and
\[
 \frac{1}{A}\int_0^1 ||\dot{\gamma}||dt\leq\mathrm{length}_{{M}_{0}} \left(\gamma\cap E\right).
\]
Combining these two inequalities we see,
\[
\frac{\mathrm{length}_{\widetilde{{M}}_{0,\rho}}\left({\gamma}\cap E \right)}{\mathrm{length}_{{M}_{0}}\left(\gamma\cap E \right)} \leq  A_1A
\]
By switching the roles of $\widetilde{{M}}_{0,\rho}$ and ${M}_{0}$ and using a similar argument we see
\[
\frac{\mathrm{length}_{{M}_{0}}\left({\gamma}\cap E \right)}{\mathrm{length}_{\widetilde{{M}}_{0,\rho}}\left(\gamma\cap E \right)} \leq AA_1
\]
Therefore, $\frac{1}{1+AA_1}\leq \frac{\widetilde{d}_{0,\rho}\left(x,y\right)}{d_0\left(x,y\right)}\leq 1+AA_1$ and so we conclude that
\begin{equation}\label{e:T3}
    d_\scF\left(\Mz,\tMz\right)\leq A\rho^2,
\end{equation}
where $A$ depends on $g$.
\end{proof}

Finally, we conclude the proof of \Cref{p :pipe} by combining \eqref{e:T1}, \eqref{e:T2}, and \eqref{e:T3} to see 
\[
d_\scF\left(\Mz,\Mp\right) \leq A{\sqrt{\rho}},
\]
where $A$ depends on $g$ and $\ell$.
\end{proof}

\subsection{The proofs of Theorems \ref{t: main} and \ref{t: mainbilip} and Corollary \ref{c: density}}
Recall by assumption $cd_g\leq d<d_g$ so the identity map is a distance decreasing biLipschitz map.

First, we will define the two more integral current spaces: $\cM^{\mathrm{tun}}_\eps$ which are associated to $M$ with tunnels attached and $\cM^{\mathrm{str}}_\eps$ which are to $M$ with strings attached. Both are inductively defined.

Let $\{p_1,\ldots,p_{N\left(\epsilon\right)}\}$ be a $\epsilon$-net of $M$. Now for all $\epsilon>0$ and $i\in\{1,2,\ldots, N\left(\epsilon\right)\}$, we choose $N\left(\epsilon\right)-1$ uniformly spaced points which lie on the geodesic sphere $\partial B_{p_i}\left(\epsilon\right)$. Denote these points as $\left\{q_j^i\right\}_{j\in\{1,\ldots,\hat{i},\ldots N\left(\epsilon\right)\}}$ where $\hat{i}$ indicates that the index $\hat{i}$ is omitted. Note we have
\[
d_g\left(q_j^i,q_{j'}^i\right)>\frac{\epsilon}{N\left(\epsilon\right)}
\]
for each $j\neq j'$ and all $i$. Now define $\left(Y_\epsilon,d_{Y_\eps}\right)$ to be the metric space where we attach line segments to $M$. Specifically, let $\ell_j^i=d\left(q^i_j, q^j_i\right)$
\[
Y_\epsilon =M \bigsqcup_{i<j}^{N\left(\eps\right)} [0,\ell_j^i],
\]
where we identify the endpoints of $[0,\ell_j^i]$ with the pair of points $\{q_j^i,q_i^j\}$. And $d_{Y_{\epsilon}}$ is the induced length metric. Finally, define
\begin{equation*}\label{e: def of stringspace}
    \cM^{\mathrm{str}}_\eps = \left(M,d_{Y_\eps},\left\llbracket M\right \rrbracket\right).
\end{equation*}
 Next, we define $\left(X_\epsilon,d_{X_\eps}\right)$ to be the metric space where we attach tunnels to $M$. Specifically, fix the radius
\[
\rho_{i,j}\left(\eps\right)\leq \min\left\{\frac{\eps}{N\left(\eps\right)^4},\rho_{0,i,j}\right\}
\]
so that $\{B_{q^i_j}\left(\rho_{i,j}\left(\eps\right)\right)\}$ are disjoint and $\rho_{0,i,j}$ is chosen using \Cref{l: useful}. Define $X_\eps$ by removing the balls $B_{q^i_j}
\left(\rho_{i,j}\left(\eps\right)\right)$, for each $i\in\{1,\ldots,N\left(\eps\right)\}$ and $j\in \{1,\ldots,\hat{i},\ldots, N\left(\eps\right)\}$, and attach 
tunnels $T_i^j$ using \Cref{p: propT} where we choose $d=\ell_j^i$ and $\delta = \delta_j^i < \rho_{i,j}\left(\eps\right)$ 
along boundaries of $B_{q^i_j}\left(\rho_{i,j}\left(\eps\right)\right)$ and $B_{q^j_i}\left(\rho_{i,j}\left(\eps\right)\right)$. Specifically,
\[
X_\eps = \left(M\setminus\bigcup_{i\neq j}^{N\left(\eps\right)} B_{q^i_j}\left(\rho_{i,j}\left(\eps\right)\right)\right) \bigcup_{i<j}^{N\left(\eps\right)} T^j_i,
\]
and call the induced Riemannian metric $g_{X_\eps}$ and call the induced distance function $d_{X_\eps}$. Moreover, by \Cref{p: propT} we can ensure that the scalar curvature of $X_\eps$ is greater than $\kappa - \frac{1}{m\left(\eps\right)}$, where $m\left(\eps\right)\xrightarrow{\eps\to0}\infty$. Finally, define
\begin{equation*}\label{e: def of tunnel space}
    \cM^{\mathrm{tun}}_\eps = \left(X_\eps,d_{X_\eps}, \left\llbracket {X_\eps}\right\rrbracket \right)
\end{equation*}

The proofs of the next two propositions are variations on \cite[Proposition 5.1 and 5.2]{BKS}.
\begin{proposition}\label{p: tunnels close to strings}
    There is a constant $A>0$ so that, for all sufficiently small $\epsilon>0$, we have
\[
d_\scF\left(\cM^\mathrm{tun}_\epsilon,\cM^\mathrm{str}_\epsilon\right) \le A\sqrt{\epsilon}.
\]
\end{proposition}

\begin{proof}
For each $\epsilon>0$, enumerate the collection of pairs 
$\{q_j^i,q_i^j\}\subset M$ from $1$ to 
$K:=\frac{N\left(\epsilon\right)\left(N\left(\epsilon\right)-1\right)}{2}$. For $k\in\{1,\ldots, K\}$,
let $\cM^\mathrm{tun}_\epsilon\left(k\right)$ denote the integral current space resulting from 
replacing the first $k$ tunnels $T_i^j$ in the construction of $\cM^\mathrm{tun}_\epsilon$ 
with the corresponding strings $[0,\ell_j^i]$ as in the construction of $\cM^\mathrm{str}_\epsilon$. 
Notice that $\cM^\mathrm{tun}_\epsilon\left(0\right)=\cM^\mathrm{tun}_\epsilon$ and $\cM^\mathrm{tun}_\epsilon\left(K\right)=\cM^\mathrm{str}_\epsilon$. 

We note that $\cM^\mathrm{tun}_\epsilon\left(k\right)$ satisfies $\diam\left(\cM^\mathrm{tun}_\epsilon\right)\leq 2\left(\diam_g\left(M\right)+\diam_d\left(M\right)\right)$ for all $k$ and $\eps$ because $|\ell_j^i|\leq \diam_d\left(M\right)$ and a path between two points of $\cM^\mathrm{tun}_\epsilon$ can be defined by first exiting a tunnel if necessary (which adds length at most $2\diam_d\left(M\right)$) and then traveling along small perturbation of a geodesic in $\left(M,g\right)$ (which adds length at most $2\diam_g\left(M\right)$).
 
Now item $\eqref{e: con2}$ of \Cref{p: propT} states there is
a constant $A>0$ so that $\vol\left(T_j^i\right)\leq A\frac{\rho_{i,j}\left(\eps\right)^{n-1}}{2^{n-1}}$, where $A$ depends on $\ell$ and $g$.
Using this and our choice of $\rho_{i,j}\left(\eps\right)$, we estimate
\begin{align*}
\vol\left(\cM^\mathrm{tun}_\epsilon\right)&\leq\vol_g\left(M\right)+\sum_{i<j}^{N\left(\eps\right)}\vol\left(T_j^i\right)\\
&\leq\vol_g\left(M\right)+N\left(\epsilon\right)^2 A \frac{\rho_{i,j}\left(\eps\right)^{n-1}}{2^{n-1}}\\
&=\vol_g\left(M\right)+A\epsilon^{n-1}.
\end{align*}
Therefore, $\vol\left(\cM^\mathrm{tun}_\epsilon\right)$ is bounded above by a constant independent of $k$ and $\epsilon$.
Uniform control of the above quantities implies that
the constant in \eqref{e :pipe} obtained by applying \Cref{p :pipe} to a ball $B_{q_j^i}\left(\rho_{i,j}\left(\eps\right)\right)$ in 
$\cM^\mathrm{tun}_\epsilon\left(k\right)$ is uniformly bounded in $k$ and $\epsilon$ by some constant $A$.

Apply \Cref{p :pipe} iteratively to get:
\begin{align*}
d_\scF\left(\cM^\mathrm{tun}_\epsilon,\cM^\mathrm{str}_\epsilon\right)&\leq\sum_{k=1}^Kd_\scF\left(\cM^\mathrm{tun}_\epsilon\left(k-1\right),
	\cM^\mathrm{tun}_\epsilon\left(k\right)\right)\\
&\leq N\left(\eps\right)^2 A\sqrt{\frac{\rho_{i,j}\left(\eps\right)}{2}}\\
&=A\sqrt{\eps}.
\end{align*}
\end{proof}

\begin{proposition}\label{p: strings close to N}
  The integral current spaces $\cM^\mathrm{str}_\epsilon$ converge to $\cM$
in the intrinsic flat sense as $\epsilon\to0$, where
\[
\cM=\left(M, d,\left\llbracket M\right\rrbracket\right)
\]
\end{proposition}

\begin{proof}
Note that the integral currents and canonical sets of the integral current spaces $\cM$ and $\cM^\mathrm{str}_\epsilon$ are identical for all $\epsilon$. We will apply \cite[Theorem A.1]{HLS} (see \Cref{t: HLS} above). Therefore, 
to show $\cM^\mathrm{str}_\epsilon\to \cM$ in the intrinsic flat sense, it suffices to show 
\[
\sup_{M\times M}|d_{Y_\eps}- d|\to0
\]
and the existence of a $\alpha>0$, independent of $\epsilon$, so that
\[
\frac{1}{\alpha}\leq\frac{d_{Y_\eps}\left(x,y\right)}{ d\left(x,y\right)}\leq\alpha
\]
for all $x,y\in M$. We will now verify the above conditions.

For $\epsilon>0$ and $x,y\in M$, choose points $p_i$ and $p_{i'}$ in the net 
$\{p_i\}_{i=1}^{N\left(\epsilon\right)}$ closest to $x$ and $y$, respectively. Now,
\begin{align*}
d_{Y_\eps}\left(x,y\right)&\leq d_g\left(x,q_{i'}^i\right)+ d\left(q_{i'}^i,q_i^{i'}\right)
	+d_g\left(q_i^{i'},y\right)\\
&\leq6\epsilon+d\left(q_{i'}^i,q_i^{i'}\right),
\end{align*}
where the first inequality follows from the fact that 
$ d\left(q_{i'}^i,q_i^{i'}\right)=d_{Y_\epsilon}\left(q_{i'}^i,q_i^{i'}\right)$
and the second inequality follows from our choice of points $q_i^j\in\partial B_{p_i}\left(\epsilon\right)$.
Now we would like to compare $ d\left(x,y\right)$ and $ d\left(q_{i'}^i,q_i^{i'}\right)$.
Note

\[
 d\left(q_{i'}^i,q_i^{i'}\right)\leq  d\left(x,q_{i'}^i\right)+ d\left(x,y\right)+ d\left(q_i^{i'},y\right).
\]
Therefore, 
\begin{equation}
     d\left(x,y\right)\geq  d\left(q_{i'}^i,q_i^{i'}\right) -6\eps
\end{equation}
Combining these inequalities, we see
\[
d_{Y_\epsilon}\left(x,y\right)- d\left(x,y\right)\leq12\epsilon.
\]
By construction $ d\left(x,y\right)\leq d_{Y_\epsilon}\left(x,y\right)$. Therefore, combining that with the above inequality we have
\[
\left|d_{Y_\eps}\left(x,y\right)-d\left(x,y\right)\right|\leq 12\eps \xrightarrow{\eps\to 0} 0.
\]
For the other condition, we notice that 
\[
1=\frac{ d\left(x,y\right)}{ d\left(x,y\right)}\leq\frac{d_{Y_\eps}\left(x,y\right)}{ d\left(x,y\right)} \leq \frac{d_{g}\left(x,y\right)}{ d\left(x,y\right)} \leq \frac{1}{c},
\]
where $\frac{1}{c}$ is the Lipschitz constant of the identity map from $\left(M,d\right)$ to $\left(M,d_g\right)$. Therefore, take $\alpha= \frac{1}{c}$ in the above condition.
\end{proof}

Now \Cref{t: mainbilip} follows readily. 
\begin{proof}[Proof of \Cref{t: mainbilip}]
    By the triangle inequality, \Cref{p: tunnels close to strings} and \Cref{p: strings close to N} we see that 
\[
	d_\scF\left(\cM^\mathrm{tun}_\eps,\cM\right) 
		\leq d_\scF\left(\cM^\mathrm{tun}_\eps,\cM^\mathrm{str}_\eps\right) + d_\scF\left(\cM^\mathrm{str}_\eps,\cM\right) 
		\to 0
\]
as $\eps \to 0$.
\end{proof}
\Cref{t: main} follows immediately from the next corollary.
\begin{corollary}
    Let $n\geq 3$ and $g$ be a Riemannian metric on the standard smooth $n$-sphere. Then there exists a sequence of Riemannian manifolds $\left(M_j,g_j\right)$ with scalar curvature $R_{g_j}>0$ such that
    \[
    d_\scF\left(\left(M_j,d_{g_j},\left\llbracket {M_j}\right\rrbracket\right)\left(\bS^n,d_{g},\left\llbracket {\bS^n}\right\rrbracket\right)\right)\xrightarrow{j\to\infty} 0.
    \]
\end{corollary}

\begin{proof}
    Consider the identity map $\left(\bS^n,g_{rd}\right)\xrightarrow{\mathrm{id}} \left(\bS^n,g\right)$ and the identity map $\left(\bS^n,g\right) \xrightarrow{\mathrm{id}} \left(\bS^n,g_{rd}\right)$. Both are diffeomorphisms so by compactness there exists a $C>1$ depending on $g_{rd}$ and $g$ such that 
    \[
  \frac{1}{C}d_{g_{rd}}\left(x,y\right)\leq d_{g}\left(x,y\right) \leq C d_{g_{rd}}\left(x,y\right).
    \]
    Therefore, the identity map $\left(\bS^n,C^2g_{rd}\right)\xrightarrow{\mathrm{id}}\left(\bS^n, g\right)$ is a distance decreasing biLipschitz map because 
    \[
    \frac{1}{C^2}\leq \frac{d_{g}\left(x,y\right)}{d_{C^2g_{rd}} \left(x,y\right)}= \frac{1}{C} \frac{d_{g}\left(x,y\right)}{d_{g_{rd}}\left(x,y\right)} \leq 1.
    \]
    Note that $R_{C^2g_{rd}} = \frac{1}{C^2} n\left(n-1\right) >\frac{n\left(n-1\right)}{2C^2}>0$. Finally, apply \Cref{t: mainbilip} to complete the proof.
\end{proof}

Finally, we prove \Cref{c: density}.
\begin{proof}[Proof of \Cref{c: density}]
    Let $M^n$ be a closed oriented smooth manifold and $g$ a Riemannian metric on $M$. Then by \cite{KW} there exists a metric $g_-$ on the smooth manifold $M$ such that the scalar curvature $R_{g_-}=-\frac{1}{2}$.  Consider the identity map $\left(M^n,g_-\right)\xrightarrow{\mathrm{id}} \left(M^n,g\right)$ and $\left(M^n,g\right)\xrightarrow{\mathrm{id}}\left(M^n,g_-\right)$. Both are diffeomorphisms so by compactness there exists a $C>1$ depending on $g$ and $g_-$ such that 
    \[
  \frac{1}{C}d_{g_{-}}\left(x,y\right)\leq d_{g}\left(x,y\right) \leq C d_{g_{-}}\left(x,y\right).
    \]
    Therefore, the identity map $\left(M^n,C^2g_{-}\right)\xrightarrow{\mathrm{id}}\left(M^n, g\right)$ is a distance decreasing biLipschitz map, i.e., 
    \[
    \frac{1}{C^2}\leq \frac{d_{g}\left(x,y\right)}{d_{C^2g_{-}} \left(x,y\right)}= \frac{1}{C} \frac{d_{g}\left(x,y\right)}{d_{g_{-}}\left(x,y\right)} \leq 1.
    \]
    Moreover, $R_{C^2g_{-}}=\frac{1}{C^2}R_{g_{-}}= -\frac{1}{2C^2}\geq -\frac{1}{2}.$
    Now apply \Cref{t: mainbilip} to complete the proof.
\end{proof}

\section{Bilipschitz maps and Pipe Filling}\label{sec: Bilip}
In this section, we will prove a slight generalization of \Cref{t: mainbilip}.
\begin{theorem}\label{t: 6.1}
    Let $\kappa\in\R$, $n\geq 3$, and $\left(M^n,g\right)$ be a closed oriented Riemannian manifold with scalar curvature $R_g\geq\kappa$. Let $d$ be a distance function such that there exists a strictly distance decreasing biLipshitz map 
    \[
    F:\left(M,d_g\right)\to \left(X,d\right)
    \]
   where $d_g$ is the induced distance function from the Riemannian metric. Then there exists a sequence of Riemannian manifolds $\left(M^n_j,g_j\right)$ with scalar curvature $R_{g_j}\geq \kappa -\frac{1}{j}$ 
   such that
     \begin{equation*}
    d_{\scF}\left(\left(M_j,d_{g_j},\left\llbracket M_j\right\rrbracket\right), \left(X,d,F_\#\left\llbracket M\right\rrbracket\right)\right)\xrightarrow{j\rightarrow\infty}0.
\end{equation*}   
\end{theorem}
This theorem will follow from \Cref{t: mainbilip} and the following lemma. First, we state a definition
\begin{deff}
     A map $F:\left(X,d_X\right)\to\left(Y,d_Y\right)$ between to metric spaces is called a $\left(c,C\right)$-biLipschitz if $F$ is biLipschitz and there are constants $c,C>0$ such that for all $p,q\in X$
     \[
     cd_X\left(p,q\right)\leq d_Y\left(F\left(p\right),F\left(q\right)\right) \leq Cd_X\left(p,q\right).
     \]
\end{deff}
\begin{lemma}\label{l: mainbilip}
    Let $\left(X, d, T\right)$ be an integral current space. Let $F:\left(X,d\right)\to \left(X,\widetilde{d}\right)$ be a $\left(c,C\right)$-biLipschitz map. Define the integral current space $\left(X,\widetilde{d},F_\#T\right)$. Define the metric space $\left(X,F^*{\widetilde{d}}\right)$, where $F^*{\widetilde{d}}\left(x,y\right)=\widetilde{d}\left(F\left(x\right),F\left(y\right)\right)$.
    
    Then $\mathrm{id}:\left(X,d\right)\to \left(X,F^*{\widetilde{d}}\right)$ is a $\left(c,C\right)$-biLipschitz map and $\left(X,F^*{\widetilde{d}},T\right)$ is an integral current space. Moreover, $F:\left(X,F^*{\widetilde{d}}, T\right)\to \left(X,\widetilde{d},F_\#T\right)$ is a current preserving isomorphism and  
    \[d_\scF\left(\left(X,F^*{\widetilde{d}}, T\right), \left(X,\widetilde{d},F_\#T\right)\right)=0.\]
\end{lemma}
\begin{proof}
    First, we check $F^*{\widetilde{d}}$ is a distance function. Let $x,y,z\in X$
    \[
    F^*{\widetilde{d}}\left(x,y\right)=\widetilde{d}\left(F\left(x\right),F\left(y\right)\right)\geq 0.
    \]
    Moreover, $x=y$ iff $F\left(x\right)=F\left(y\right)$ since $F$ is a bijection. Therefore,
    \[
    F^*{\widetilde{d}}\left(x,y\right)=0 \iff x=y.
    \]
    $F^*{\widetilde{d}}$ is symmetric:
    \[
    F^*{\widetilde{d}}\left(x,y\right) = \widetilde{d}\left(F\left(x\right),F\left(y\right)\right)= \widetilde{d}\left(F\left(y\right),F\left(x\right)\right)=F^*{\widetilde{d}}\left(y,x\right).
    \]
   Note that as $F$ is a bijection we have for any $w\in X$ there exists a $z\in X$ such that $F\left(z\right)=w$. Lastly, we check the triangle inequality. 
    \begin{align*}
         F^*{\widetilde{d}}\left(x,y\right) &= \widetilde{d}\left(F\left(x\right),F\left(y\right)\right) \leq \widetilde{d}\left(F\left(x\right),w\right) + \widetilde{d}\left(w,F\left(y\right)\right) = \widetilde{d}\left(F\left(x\right),F\left(z\right)\right) + \widetilde{d}\left(F\left(z\right),F\left(y\right)\right) \\&= F^*{\widetilde{d}}\left(x,z\right) +F^*{\widetilde{d}}\left(z,y\right).
    \end{align*}
    Therefore, $F^*{\widetilde{d}}$ is a distance function.
    Let $x,y\in X$, then 
    \[
  c d\left(x,y\right) \leq F^*{\widetilde{d}}\left(x,y\right)=\widetilde{d}\left(F\left(x\right),F\left(y\right)\right) \leq C d\left(x,y\right).
    \]
    Thus, $\mathrm{id}:\left(X,d\right)\to \left(X,F^*{\widetilde{d}}\right)$ is a $\left(c,C\right)$-biLipschitz map and $\left(X,F^*{\widetilde{d}},T\right)$ is an integral current space.
    Let $x,y\in X$ and observe $F:\left(X,F^*{\widetilde{d}}, T\right)\to \left(X,\widetilde{d},F_\#T\right)$ is an isometry since
    \[
    F^*{\widetilde{d}}\left(x,y\right)=\widetilde{d}\left(F\left(x\right),F\left(y\right)\right).
    \]
    Moreover, $F:\left(X,F^*{\widetilde{d}}, T\right)\to \left(X,\widetilde{d},F_\#T\right)$ is current preserving.
\end{proof}
\begin{proof}[Proof of \Cref{t: 6.1}]
     Let $\kappa\in\R$, $n\geq 3$, and $\left(M^n,g\right)$ be a closed oriented Riemannian manifold with scalar curvature $R_g\geq\kappa$. Let $d$ be a distance function such that there exists a strictly distance decreasing biLipshitz map 
    \[
    F:\left(M,d_g\right)\to \left(X,d\right)
    \]
   where $d_g$ is the induced distance function from the Riemannian metric. Therefore, there exists constants $c\leq C<1$ such that 
   \[
     cd_g\left(p,q\right)\leq d\left(F\left(p\right),F\left(q\right)\right) \leq Cd_g\left(p,q\right).
     \]
   Consider the following integral current spaces $\left(X,d,F_\#\left\llbracket M\right\rrbracket\right)$ and $\left(M,F^*d,\left\llbracket M\right\rrbracket\right)$. By \Cref{l: mainbilip}, we have that $F^*d:M\times M\to \bR$ satisfies
    \[
    cd_g\left(x,y\right)\leq F^*d\left(x,y\right) \leq Cd_g\left(x,y\right) < d_g\left(x,y\right).
    \]
    Therefore, using \Cref{t: mainbilip} we see that there exists a sequence of Riemannian manifolds $\left(M^n_j,g_j\right)$ with scalar curvature $R_{g_j}\geq \kappa -\frac{1}{j}$
   such that
   \[d_{\scF}\left(\left(M_j,d_{g_j},\left\llbracket M_j\right\rrbracket\right), \left(M,F^*d,\left\llbracket M\right\rrbracket\right)\right)\xrightarrow{j\rightarrow\infty}0. \]
  Finally,  we have
     \begin{align*}
         d_{\scF}\left(\left(M_j,d_{g_j},\left\llbracket M_j\right\rrbracket\right), \left(X,d,F_\#\left\llbracket M\right\rrbracket\right)\right) &\leq  d_{\scF}\left(\left(M_j,d_{g_j},\left\llbracket M_j\right\rrbracket\right), \left(M,F^*d,\left\llbracket M\right\rrbracket\right)\right)\\
         &\quad\quad+  d_{\scF}\left(\left(M,F^*d,\left\llbracket M\right\rrbracket\right), \left(X,d,F_\#\left\llbracket M\right\rrbracket\right)\right) \\
         &= d_{\scF}\left(\left(M_j,d_{g_j},\left\llbracket M_j\right\rrbracket\right), \left(M,F^*d,\left\llbracket M\right\rrbracket\right)\right)\xrightarrow{j\rightarrow\infty}0, 
     \end{align*}
     where in the final equality, we use the conclusion of \Cref{l: mainbilip} that $F:\left(M, F^*d,\left\llbracket M\right\rrbracket\right) \to \left(X, d, F_\#\left\llbracket M\right\rrbracket\right)$ is a current-preserving isometry.
\end{proof}

\bibliographystyle{unsrt}
\bibliography{refs}
\end{document}